\documentclass[10pt]{amsart}

\usepackage{amsthm}
\usepackage{amssymb}
\usepackage{amsmath}
\usepackage{mathtools}
\usepackage{pdfpages}
\usepackage{float}
\usepackage[colorlinks=true,pdftex,unicode=true,linktocpage,bookmarksopen,hypertexnames=false]{hyperref}
\usepackage{tikz-cd}
\usepackage{caption}
\usepackage[capitalise]{cleveref}

\newcommand{\Ann}{\operatorname{Ann}}
\newcommand{\Soc}{\operatorname{Soc}}
\newcommand{\Aut}{\operatorname{Aut}}

\newcommand{\id}{\operatorname{id}}
\newcommand{\Ret}{\operatorname{Ret}}

\newcommand{\Z}{\mathbb{Z}}

\DeclareMathOperator{\ab}{ab}

\makeatletter
\numberwithin{equation}{section}
\numberwithin{figure}{section}
\numberwithin{table}{section}
\newtheorem{thm}{Theorem}[section]
\newtheorem*{thm*}{Theorem}
\newtheorem{lem}[thm]{Lemma}
\newtheorem{cor}[thm]{Corollary}
\newtheorem{pro}[thm]{Proposition}
\newtheorem{notation}[thm]{Notation}

\theoremstyle{definition} 
\newtheorem{defn}[thm]{Definition}

\newtheorem{rem}[thm]{Remark}
\newtheorem{exa}[thm]{Example}

\makeatother

\title{Isoclinism of skew braces}
\author{T. Letourmy}
\author{L. Vendramin}

\address[T. Letourmy, L. Vendramin]{
Department of Mathematics, Vrije Universiteit Brussel, Pleinlaan 2, 1050 Brussel, Belgium}
\email{Leandro.Vendramin@vub.be}

\address[T. Letourmy]{D\'epartement de Math\'ematique, Universit\'e Libre de Bruxelles, Boulevard du Triomphe, B-1050 Brussels, Belgium}
\email{thomas.letourmy@ulb.be}


\begin{document}

\begin{abstract}
    We define isoclinism of skew braces and
    present several applications. We study 
    some properties of skew braces
    that are invariant under isoclinism. For example,
    we prove that 
    right nilpotency is an isoclinism invariant. 
    This result has application in 
    the theory of set-theoretic solutions to the Yang--Baxter
    equation. 
    We define isoclinic solutions and study 
    multipermutation solutions under isoclinism. 
\end{abstract}
 
\maketitle

\section{Introduction}

The fundamental problem of constructing combinatorial solutions 
of the Yang--Baxter equation (YBE) is nowadays 
based on the use of specific (associative and non-associative) 
algebraic structures associated with the solution. 
One particular algebraic structure stands out: skew braces. 
Such structures were introduced in \cite{MR2278047} and \cite{MR3647970}.
The theory originates in Jacobson radical rings, and it is now
considered a hot topic, as skew braces appear in several 
different areas of mathematics, see for 
example \cite{MR4390798,MR3291816,MR3763907}.  

Skew braces classify solutions \cite{MR3835326,MR3527540}. 
This justifies the search for classification results on
skew braces. In this vein, several results are known; see for example \cite{MR4113853,MR3647970}. 
However, the classification of skew braces of order $p^n$ ($p$ prime)  
is still hard to achieve. 

To classify $p$-groups, in  \cite{MR3389} 
Hall introduced a specific equivalence relation that simplifies the problem. Without technicalities, 
one defines isoclinic groups
as groups that have ``essentially the same" commutator functions. Recall
that the commutator function for 
the group $G$ is the
map
\[
G/Z(G)\times G/Z(G)\to G,\quad
(xZ(G),yZ(G))\mapsto [x,y]=xyx^{-1}y^{-1}.
\]
Isoclinism is an equivalence relation that generalizes isomorphisms.

There are several different motivations to consider 
this notion in the context of skew braces. 
With isoclinism of skew braces,
we will have a new tool to classify  
finite skew braces of prime-power order, something that 
ultimately will have
applications in other branches of mathematics, for example
in the theory of pre-Lie algebras \cite{MR4391819,MR4353236,MR4484785}. 

In \cite{LV}, the authors defined Schur covers of finite 
skew braces 
and proved that any two Schur covers are isoclinic. 

In this work, we define isoclinism in the context of skew braces. 
We show that the following properties are preserved under isoclinism: triviality (Theorem \ref{thm:trivial}), two-sidedness (Theorem \ref{thm:2sided}), right nilpotence (Theorem \ref{thm:right_nilpotent}), the lattice of sub skew braces containing the annihilator (Proposition \ref{pro:lattice}), and the size of $\lambda-$ and $\rho-$orbits up to a constant factor (Theorem \ref{thm:Horbits}).
As an application, we define a notion of isoclinism of set-theoretic solutions to the YBE. 
This new equivalence relation defined on the space of solutions suggests a way of attacking the classification problem of set-theoretic
solutions to the YBE. Isoclinism classes of solutions could be relevant in the study of multipermutation solutions \cite{MR3815290,MR4457900,MR1722951,MR2885602,MR3974961}. We
prove in Section \ref{section:YB} that
if $(X,r)$ 
and $(Y,s)$ are isoclinic set-theoretic solutions and 
$(X,r)$ is multipermutation, then 
$(Y,s)$ is multipermutation (see Theorem \ref{thm:YB}).

\section{Isoclinism of skew braces}

Recall that a \emph{skew brace} is a triple $(A,+,\circ)$, 
where $(A,+)$ and $(A,\circ)$ are groups such that
the compatibility condition 
$a\circ (b+c)=a\circ b-a+a\circ c$ holds for all $a,b,c\in A$. The inverse of an element $a\in A$ with respect to the circle 
operation $\circ$ will be denoted by $a'$. We will denote the groups $(A,+)$ and $(A,\circ)$ respectively by $A_+$ and $A_\circ$.

Let $A$ be a skew brace. There are two canonical actions by automorphism
\begin{align*}
    &\lambda\colon A_\circ\to \Aut(A_+), &&\lambda_a(b)=-a+a\circ b,\\
    &\rho\colon A_\circ\to \Aut(A_+), &&\rho_a(b)=a\circ b-a.
\end{align*}

By definition, a sub skew brace $B$ of $A$
is an ideal of $A$ if $B_+$ is normal in $A_+$, 
$B_\circ$ is normal in $A_\circ$ and
$\lambda_a(B)\subseteq B$ for all $a\in A$. 

We write $A^{(2)}=A*A$ to denote the additive subgroup of $A$ generated by $\{a*b:a,b\in A\}$, where $a*b=-a+a\circ b-b$ for all $a,b\in A$. One proves that $A^{(2)}$ is an ideal of $A$. If $a,b\in A$, 
we write $[a,b]_+=a+b-a-b$ and $[a,b]_\circ=a\circ b\circ a'\circ b'$. We write $[A,A]_+$ to
denote the commutator subgroup of the additive group of $A$. 
The \emph{socle} of $A$ is the ideal 
$\Soc(A)=\ker\lambda\cap Z(A,+)$ and 
the \emph{annihilator} of $A$ is
the ideal $\Ann(A)=\Soc(A)\cap Z(A,\circ)$; see
\cite{MR3917122}.

\begin{defn}
    Let $A$ be a skew brace. The \emph{commutator} $A'$ of $A$ is
    the additive subgroup of $A$ generated by $[A,A]_+$ and $A^{(2)}$. 
\end{defn}

\begin{pro}
    Let $A$ be a skew brace. Then $A'$ is an ideal of $A$.
\end{pro}

\begin{proof}
    It follows from the fact that 
    for all $a\in A$ and $x\in A'$
    $\lambda_a(x)=a*x+x\in A'$, $a+x-a=x+[-x,a]_+\in A'$ and 
    \begin{align*}
     a\circ x\circ a'=\lambda_a(x)&+[-a\circ x,a]_+\\&-a'*a+(a\circ x)*a'+[a',-a\circ x\circ a'+a\circ x]_+\in A'
    \end{align*}
    for all $a\in A$ and $x\in A'$. 
\end{proof}

A skew brace $A$ is said to be \emph{trivial} 
if $a+b=a\circ b$ for all $a,b\in A$. 

\begin{exa}
    In the case of a trivial brace $A$, $A'$ 
    is the commutator of the underlying group.
\end{exa}
\begin{notation}
In general, if $A/B$ is a quotient of a skew brace, we will denote the equivalence class of $a\in A$ in $A/B$ by $\overline{a}$.
\end{notation}
\begin{rem}
    $A'$ is the smallest ideal of $A$ 
    such that $A^{\ab}=A/A'$ is an abelian group. 
\end{rem}

\begin{lem}
We define two collections of maps $\phi_+$ and $\phi_*$ that associate respectively to every skew brace $B$ the maps
\begin{align}
    \label{eq:commutator}&\phi^B_{+}\colon (B/\Ann B)^2\to B',&&(\overline{a},\overline{b})\mapsto [a,b]_+,\\
    \label{eq:star}&\phi^B_{*}\colon (B/\Ann B)^2\to B',&&(\overline{a},\overline{b})\mapsto a*b. 
\end{align}
    The maps \eqref{eq:commutator} and \eqref{eq:star} are well-defined. 
\end{lem}

\begin{proof}
    A direct calculation shows that \eqref{eq:commutator} is well-defined. To prove
    that \eqref{eq:star} is well defined, let $\left(\overline{a},\overline{b}\right)=\left(\overline{c},\overline{d}\right)\in (B/\Ann B)^2$. We have 
    \begin{align*}
    -a+a\circ b-b-\left(-c+c\circ d-d\right)
    = -a+a\circ d-c\circ d +c.
\end{align*}
After conjugating by $c$, we get 
\begin{align*}
    c-a+a\circ d-c\circ d &=  \left(c-a\right)\circ a\circ d-c\circ d\\
    &=  \left(\left(c-a\right)+ a\right)\circ d-c\circ d\\
    &=  c\circ d-c\circ d\\
    &= 0.
\end{align*}
Since the conjugation by $c$ is an automorphism, we have $a*b=c*d$.
\end{proof}

\begin{defn}
    We say that the braces $A$ and $B$ are \emph{isoclinic} 
    if there are two isomorphisms 
    $\xi\colon A/\Ann A\to B/\Ann B$ and $\theta\colon A'\to B'$ 
    such that 
    \begin{equation}
        \label{eq:isoclinism}
    \begin{tikzcd}
	{A'} & {(A/\Ann A)^2} & {A'} \\
	{B'} & {(B/\Ann B)^2} & {B'}
	\arrow["\phi_+^A"', from=1-2, to=1-1]
	\arrow["{\theta }", from=1-1, to=2-1]
	\arrow["{\phi_+^B}", from=2-2, to=2-1]
	\arrow["{\phi_*^A}", from=1-2, to=1-3]
	\arrow["{\phi_*^B}"', from=2-2, to=2-3]
	\arrow["\theta", from=1-3, to=2-3]
	\arrow["\xi\times\xi", from=1-2, to=2-2]
    \end{tikzcd}
    \end{equation}
    commutes. We call the pair $(\xi,\theta)$ a skew brace \emph{isoclinism}.
\end{defn}
\begin{notation}
    For $A$ a skew brace and $B$ and $C$ 
    sub skew braces, let 
   \[
   B+C=\lbrace b+c: b\in B, c\in C\rbrace.
   \]
\end{notation}

Note that, in general, 
if $B$ and $C$ are sub skew braces of $A$, then 
$B+C$ is not a skew brace. However, 
$B+C$ is always a skew brace if $C$ is an ideal.

\begin{pro}
\label{pro:lattice}
    Let $A$ and $B$ be isoclinic skew braces. Then sub skew braces of $A$ containing
    $\Ann(A)$ are in bijective correspondence with sub skew 
    braces of $B$ containing $\Ann(B)$. 
    Furthermore, these corresponding sub skew braces are isoclinic. 
\end{pro}

\begin{proof}
    Let $(\xi,\theta)$ be an isoclinism between $A$ and $B$. The correspondence is given by $A_1\mapsto \pi^{-1}(\xi(A_1/\Ann(A))$ for all skew braces $\Ann(A)\subseteq  A_1\subseteq A$, where the map $\pi\colon B\to B/\Ann(B)$ is the canonical map. Let $\Ann(A)\subseteq A_1\subseteq A$ and $\Ann(B)\subseteq B_1\subseteq B$ be corresponding skew braces. Then, by the commutativity of the diagram \eqref{eq:isoclinism}, $\theta$ restricts to an isomorphism $A_1'\to B_1'$ and $\xi$ factors through an isomorphism $\overline{A_1}\to \overline{B_1}$. These two maps form an isoclinism from $A_1$ to $B_1$.
\end{proof}

\begin{pro}
    Let $A$ be a skew brace and $K$ be a sub skew brace of $A$. 
    Then $K$ is isoclinic to $K+\Ann(A)$. 
\end{pro}

\begin{proof}
    It is a direct consequence of the fact that $K'=(K+\Ann(A))'$ and that 
    $K/\Ann(K)\simeq (K+\Ann(A))/\Ann(K+\Ann(A))$.
\end{proof}

\begin{pro}
    Let $A$ be a skew brace and $K$ be a sub skew brace of $A$. 
    If $A/\Ann(A)$ is finite, 
    then $A$ is isoclinic to $K$ if and only if $K+\Ann(A)=A$.
\end{pro}

\begin{proof}
    Assume $A$ is isoclinic to $K$. Then 
    \[|K/\Ann K|=|A/\Ann A|\geq |K/K\cap \Ann(A)|\geq |K/\Ann K|.
    \]
    Therefore, $(K+\Ann(A))/\Ann(A)=A/\Ann(A)$. Thus, $K+\Ann(A)=A$.
\end{proof}

\begin{pro}
    Let $A$ and $B$ be isoclinic skew braces with isoclinism $(\xi,\theta)$. 
    If $K\subseteq A'$ is an ideal of $A$, then $\theta(K)\subseteq B'$ is an ideal of $B$. In addition, $A/K$ is isoclinic to $B/\theta(K)$.
\end{pro}

\begin{proof}
    First notice that for all elements $x\in A'$ we have $\xi(\overline{x})=\overline{\theta(x)}$. It is enough to verify it on the generators of $A'$. The commutativity of \eqref{eq:isoclinism} implies that for all $a,a_1\in A$, 
    \[\overline{\theta([a,a_1]_+)}= [\xi(\overline{a}),\xi(\overline{a_1})]_+,\quad
    \overline{\theta(a*a_1)}= \xi(\overline{a})*\xi(\overline{a_1}). 
    \]
    
    We show that $\theta(K)$ is an ideal of $B$. Clearly $\theta(K)_+$ is a subgroup of $B_+$ and $\theta(K)_\circ$ is a subgroup of $B_\circ$. Let $k\in K$. Let $b\in B$ and $a\in A$ be 
    such that 
    $\xi(\overline{a})=\overline{b}$. The commutativity 
    of the diagram \eqref{eq:isoclinism} implies that $\theta([a,k]_+)=[b,\theta(k)]_+$ and $\theta(a*k)=b*\theta(k)$. Thus, $b+\theta(k)-b=\theta(a+k-a)\in \theta(K)$ and $\lambda_b(\theta(k))=\theta(\lambda_a(k))$. Therefore $\theta(K)$ is a left ideal of $B$ that
    is normal in the additive group of $B$. Similarly, one shows that $K_\circ$ is a normal subgroup of $B_\circ$ using the commutativity of the diagram \eqref{eq:isoclinismBis}. 
    It is left to see that $A/K$ is isoclinic to $B/\theta(K)$. Firstly, $(A/K)'=A'/K$ and $(B/\theta(K))'=B'/\theta(K)$. Thus $\theta$ factors through an isomorphism $\overline{\theta}\colon (A/K)'\to (B/\theta(K))'$. The annihilator of $A/K$ corresponds by the canonical homomorphism to an ideal $Q$ of $A$ such 
    that 
    \[
    \Ann(A)+K\subseteq Q\subseteq A.
    \]
    Similarly, the annihilator of $B/\theta(K)$ corresponds to an ideal $Q_1$ of $B$ such that 
    \[
    \Ann(B)+\theta(K)\subseteq Q_1\subseteq B.
    \]
    Thus 
    \begin{align*}
    &\overline{A/K}\simeq A/Q\simeq \overline{A}/(Q/\Ann(A)),
    &\overline{B/\theta(K)}\simeq B/Q_1\simeq \overline{B}/(Q_1/\Ann(B)).
    \end{align*}
    
    It is left to see that $\xi(Q/\Ann(A))=Q_1/\Ann(B)$. Let $q\in Q$, then $[q,x]_+\in K$, $q*x\in K$ and $[q,x]_\circ\in K$ for all $x\in A$. Let $q_1\in B$ be a representative of $\xi(\overline{q})$. By the commutativity of the diagrams \eqref{eq:isoclinism} and \eqref{eq:isoclinismBis}, $[q_1,x]_+\in\theta(K)$, $q_1*x\in\theta(K)$ and $[q_1,x]_\circ\in \theta(K)$ for all $x\in B$. By a symmetric argument, 
    $Q_1/\Ann(B)\subseteq\xi(Q/\Ann(A))$. The commutativity of \eqref{eq:isoclinism} is straightforward.
\end{proof}

\begin{pro}
    \label{pro:kinta}
    Let $A$ be a skew brace and $K$ an ideal of $A$. Then $A/K$ is isoclinic to $A/(K\cap A')$.
\end{pro}

\begin{proof}
    We have that $(A/K)'= (A'+K)/K\simeq A'/(K\cap A')=(A/(K\cap A'))'$. The annihilator of $A/K$ corresponds 
    to an ideal $Q$ of $A$ such that
    $\Ann(A)+K\subseteq Q\subseteq A$. Similarly, 
    the annihilator of $A/(K\cap A')$ corresponds to an ideal $Q_1$ of $A$ such that 
    $\Ann(A)+K\cap A'\subseteq Q_1\subseteq A$. Thus $\overline{A/K}\simeq A/Q$ and $\overline{A/K\cap A'}\simeq A/Q_1$. It is left to see that $Q=Q_1$. It is clear that $Q_1\subseteq Q$. Let $q\in Q$. Then $[q,x]_+\in K$, $q*x\in K$ and $[q,x]_\circ\in K$ for all $x\in A$. In addition, by the definition of $A'$, $[q,x]_+\in A'$, $q*x\in A'$ and $[q,x]_\circ\in A'$ for all $x\in A$. Thus $q\in Q_1$. The commutativity of \eqref{eq:isoclinism} is straightforward.
\end{proof}

\begin{cor}
    Let $A$ be a finite skew brace and $K$ an ideal of $A$. 
    Then $A$ is isoclinic to $A/K$ if and only if $K\cap A'= \lbrace 0\rbrace$.
\end{cor}

\begin{proof}
    If $K\cap A'\neq \lbrace 0\rbrace$, then 
    \[|(A/K)'|=|A'/(K\cap A')|=|A'|/|K\cap A'|<|A'|. 
    \]
    Thus $A$ cannot be isoclinic to $A/K$.
    The other implication is a direct consequence of
    Proposition \ref{pro:kinta}.
\end{proof}

\begin{rem}
    Let $A$, $A_1$, $B$ and $B_1$ be skew braces. If $A$ is isoclinic to $A_1$ and $B$ is isoclinic to $B_1$, then $A\times B$ is isoclinic to $A_1\times B_1$.
\end{rem}

The following definition is motivated by 
its group-theoretic analog. 

\begin{defn}
    A skew brace $A$ such that $\Ann(A)\subseteq A'$ 
    will be called a \emph{stem skew brace}. 
\end{defn}

\begin{pro}
    Two isoclinic stem skew braces have the same order.
\end{pro}

\begin{proof}
    Let $A$ and $B$ be isoclinic stem skew braces with isoclinism $(\xi,\theta)$. As $\Ann(A)\subseteq A'$, $\overline{\theta(x)}=\xi(\overline{x})$ for all $x\in \Ann(A)$. Thus, $\theta(\Ann(A))\subseteq \Ann(B)$. A symmetric argument shows that $\theta$ restricts to an isomorphism of $\Ann(A)$ and $\Ann(B)$. The statement is then a consequence of $A/\Ann(A)\simeq B/\Ann(B)$.
\end{proof}

\begin{thm}
\label{thm:stem}
    Every skew brace is isoclinic to a stem skew brace.
\end{thm}

\begin{proof}
    Let $A$ be a skew brace and $I$ a set of indices (possibly uncountable). Let $\lbrace \xi_i:i\in I\rbrace$ be a set of generators of $A$ (as a skew brace). Let $G$ be the free abelian group generated by $\lbrace \eta_i:i\in I\rbrace$. Then $A$ is isoclinic to $A\times G$. Let $A_1\subseteq A\times G$ be the skew brace generated by $\lbrace (\xi_i,\eta_i):i\in I\rbrace$. Then clearly $A_1+G$ coincides with $A\times G$. Since $G$ is contained in the annihilator of $A\times G$, it follows that $A_1$ is isoclinic to $A\times G$. Thus $A_1$ is isoclinic to $A$. Let $\pi\colon A_1\to G$ denote the projection on the second coordinate. Then $\pi$ factors through a morphism of abelian groups $\overline{\pi}\colon A_1/A_1'\to G$. By the universal property of the free abelian group $G$, there is a homomorphism $\phi\colon G\to A_1/A_1'$ that maps $\eta_i$ to $\overline{(\xi_i,\eta_i)}$ for all $i\in I$. In fact, $\phi$ is the inverse map of $\overline{\pi}$. Therefore $A_1/A_1'\simeq G$. Moreover, 
    \[
    \Ann(A_1)/(\Ann(A_1)\cap A_1')\simeq (\Ann(A_1)+A_1')/A_1',
    \]
    that is $\Ann(A_1)/(\Ann(A_1)\cap A_1')$ is isomorphic to a subgroup of a free abelian group. Thus $\Ann(A_1)$ is the direct product of $\Ann(A_1)\cap A_1'$ with another group $K$. Since $K\cap A_1'=\lbrace 0\rbrace$, the skew brace $A_1/K$ is isoclinic to $A_1$. The annihilator of $A_1/K$ is $\Ann(A_1)/K$ and is contained in the commutator, which is $(A_1'+K)/K$.
\end{proof}

For $n\geq2$, let $C_n$ be the cyclic group of order $n$.

\begin{notation}
    Let $n\geq 2$ be an integer. Let $d$ be an integer such that $d|n$ and every prime divisor of $n$ divides $d$. We will denote by $C(n,d)$ the skew brace $C_n$ with multiplication defined as $x\circ y= x+dxy+y$. 
\end{notation}

The skew braces $C(n,d)$ appear in the work of Rump \cite{MR2298848}. 

\begin{exa}
Let $m,n>2$ be integers,
We have that $a*b=m^{n-1}ab$ for all $a,b\in C(m^n,m^{n-1})$.
    The commutator is the set of multiples of $m^{n-1}$ and the annihilator is the set of multiples of $m$. They coincide if and only if $n=2$. Thus $C(m^2,m)$ is a stem skew brace. 
    It is straightforward to see that $C(m^2,m)$ is isoclinic to 
    $C(m^n,m^{n-1})$ for all $n>1$.
\end{exa}

\section{Isoclinism Invariants}
\begin{thm}
\label{thm:trivial}
    Let $A$ and $B$ be skew braces. If $A$ and $B$ are isoclinic then $A$ is trivial
    if and only if $B$ is trivial.
\end{thm}

\begin{proof}
    It is a direct consequence of the commutativity of the right part of the diagram \eqref{eq:isoclinism}.
\end{proof}

\begin{rem}
    Two trivial skew braces are isoclinic if and only if their underlying groups are isoclinic. Therefore,
we will not distinguish between the two notions when dealing with trivial skew braces.
\end{rem}

\begin{rem}
    Let $A$ and $B$ be isoclinic skew braces. The quotients 
    $A/\Ann(A)$ and $B/\Ann(B)$ are isomorphic. It follows
    that $A$ is annihilator nilpotent if and only if $B$ is 
    annihilator nilpotent. 
\end{rem}

A skew brace $A$ is said to be \emph{right nilpotent} 
if $A^{(n)}=\{0\}$ for some $n$, where 
$A^{(1)}=A$ and $A^{(k+1)}=A^{(k)}*A$ for all $k$. 

\begin{thm}
\label{thm:right_nilpotent}
Let $A$ and $B$ be isoclinic skew braces. 
Then $A$ is right nilpotent if and only if $B$ is right nilpotent.
\end{thm}

\begin{proof}
    Let $(\xi,\theta)$ be an isoclinism between 
    $A$ and $B$. 
    Assume that $A$ is right nilpotent. By \cite[Lemma 2.5]{MR3957824}, 
    $A/\Soc(A)$ is right nilpotent. 
    
    We claim that 
    \begin{equation}
    \label{eq:Soc/Ann}
        \xi(\Soc(A)/\Ann(A))=\Soc(B)/\Ann(B).
    \end{equation}
    Let $a\in \Soc(A)$
    be such that $\xi(a+\Ann(A))=b+\Ann(B)$ for some $b\in B$. Let $a_1\in A$ and $b_1\in B$ such that $\xi(a_1+\Ann(A))=b_1+\Ann(B)$. By the commutativity of \eqref{eq:isoclinism}, 
    \begin{align*}
    b*b_1&=\phi_*^B(b+\Ann(B),b_1+\Ann(B))\\
    &=\theta(\phi_*^A(a+\Ann(A),a_1+\Ann(A)))\\
    &=\theta(a*a_1)=0.
    \end{align*}
    Similarly, $b$ is central in $(B,+)$. Now \eqref{eq:Soc/Ann} follows. 

    It follows from the first isomorphism theorem that 
    \[
    (A/\Ann(A))/(\Soc(A)/\Ann(A))\simeq (B/\Ann(B))/(\Soc(B)/\Ann(B)). 
    \]
    Hence $A/\Soc(A)\simeq B/\Soc(B)$. Thus $B/\Soc(B)$ is right nilpotent.  
    By \cite[Proposition 2.17]{MR3957824}, $B$ is right nilpotent.    
\end{proof}

We now present the notions of terms and term functions from universal algebra in the context of skew braces. A more general approach can be found in \S10, and \S11 of the book \cite{sankappanavar1981course}.

     Let $X$ be a set of objects called variables. 
     The set of \emph{skew brace terms} over $X$ is the smallest set $T(X)$ such that
\begin{enumerate}
    \item $X \cup \lbrace 0\rbrace \subseteq T(X)$, where $0$ is a formal element, and 
    \item if $p_1,p_2 \in T(X)$, then the ``strings" 
    $p_1\circ p_2,p_1+p_2,-p_1,p_2'\in T(X)$.
\end{enumerate}

For $p\in T(X)$ we write $p$ as $p(x_1, \dots, x_n)$ 
to indicate that the variables occurring in $p$ are among $x_1, \dots, x_n$. We say that a skew brace term $p$ is $n$-ary if the number of variables appearing explicitly in $p$ is $\leq n$.

Given a skew brace term $p(x_1,\dots,x_n)$ over some 
set of variables $X$ and 
a skew brace $A$, we define a map $p^A \colon A^n \to A$ inductively
as follows:
\begin{enumerate}
    \item if $p$ is a variable $x_i\in X$, then $p^A(a_1, \dots, a_n) = a_i$
        for all $a_1, \dots, a_n \in A$;
    \item if $p$ is of the form $p_1(x_1,\dots, x_n)\star p_2(x_1,\dots, x_n)$, then
        \begin{align*}
        &p^A(a_1, \dots, a_n) = p^A_1(a_1, \dots, a_n)\star p^A_2(a_1,\dots, a_n),
        \end{align*}
        where $\star$ denotes either $\circ$ or $+$;  
    \item if $p$ is of the form $p_1(x_1,\dots, x_n)'$, then 
        $p^A(a_1, \dots, a_n) = p^A_1(a_1, \dots, a_n)'$;
   \item if $p$ is of the form $-p_1(x_1,\dots, x_n)$, then 
        $p^A(a_1, \dots, a_n) = -p^A_1(a_1, \dots, a_n)$. 
\end{enumerate}

\begin{exa}
    Let $X=\lbrace x,y\rbrace$, then $p(x,y)=-x+x\circ y -y$ is a skew brace term over $X$. Given a skew brace $A$, its term function over $A$ is the map 
    $p^A\colon A^2\to A$, $(a,b)\mapsto a*b$. 
\end{exa}

\begin{notation}
   Let $A$ be a skew brace. 
   From now on, we denote by $\overline{A}$ the quotient $A/\Ann A$.
\end{notation}

\begin{lem}
\label{lem: natrans}
Let $n,m$ be integers, $\eta_1,\dots, \eta_{2m}$ be $n$-ary skew brace terms and $p$ be an $m$-ary skew brace term. Let $\phi_1,\dots, \phi_m$ be collections of maps where each $\phi_i$ is either $\phi_+$ or $\phi_*$. Then one can construct a collection of well-defined maps $\phi$ that associates to every skew brace $B$ a map $\phi^B\colon \overline{B}^n\to B'$ such that
    \begin{equation*}
      \phi^B(\overline{b_1},\dots, \overline{b_n})= 
      p^B\left(a_1(\overline{b_1},\dots,\overline{b_n}),\dots , a_m(\overline{b_1},\dots,\overline{b_n})\right),
    \end{equation*}
    where $a_i\colon \overline{B}^n\mapsto B$ is the map 
    \[
    (\overline{b_1},\dots,\overline{b_n})\mapsto \phi_i^B(\eta^{\overline{B}}_{2i-1}(\overline{b_1},\dots, \overline{b_n}),\eta^{\overline{B}}_{2i}(\overline{b_1},\dots, \overline{b_n}))
    \]
    for all $1\leq i\leq m$. In addition, if $A$ and $B$ are two isoclinic skew braces with isoclinism $(\xi,\theta)$,  then the following diagram
    \begin{equation}
\begin{tikzcd}
	{\overline{A}^n} & {A'} \\
	{\overline{B}^n} & {B'}
	\arrow["{\phi^A}", from=1-1, to=1-2]
	\arrow["{\phi^B}"', from=2-1, to=2-2]
	\arrow["\xi^n"', from=1-1, to=2-1]
	\arrow["\theta", from=1-2, to=2-2]
\end{tikzcd}
\end{equation}
    commutes.
\end{lem}

\begin{proof}
Let $B$ be a skew brace. One has that $\phi^B$ 
 is the composition of the maps
 \[\begin{tikzcd}
	{\overline{B}^n} & {\overline{B}^{2m}} & {B^m} & B
	\arrow["{\theta}", from=1-1, to=1-2]
	\arrow["{\gamma}", from=1-2, to=1-3]
	\arrow["{p^B}", from=1-3, to=1-4]
\end{tikzcd}\]
where  
\begin{align*}
\theta(\overline{b_1},\dots,\overline{b_n}) &= (\eta^{\overline{B}}_1(\overline{b_1},\dots,\overline{b_n}),\dots,\eta^{\overline{B}}_{2m}(\overline{b_1},\dots,\overline{b_n})),\\
\gamma(\overline{f_1},\dots, \overline{f_{2m}})&=(\phi^B_1(\overline{f_1},\overline{f_2}),\dots,\phi^B_m(\overline{f_{2m-1}},\overline{f_{2m}})).
\end{align*}
Suppose that $A$ and $B$ are isoclinic skew braces with isoclinism $(\xi,\theta)$.
Since $(\xi,\theta)$ is an isoclinism, \cite[Theorem 10.3]{sankappanavar1981course} implies that for all $a_1,\dots, a_n \in A$ and $1\leq i\leq m$,
\begin{align*}
    \phi_i^B(\eta^{\overline{B}}_{2i-1}(&\xi(\overline{a_1}),\dots, \xi(\overline{a_n})),\eta^{\overline{B}}_{2i}(\xi(\overline{a_1}),\dots, \xi(\overline{a_n})))\\
    &=\mathcal\phi_i^B(\xi(\eta^{\overline{A}}_{2i-1}(\overline{a_1},\dots, \overline{a_n})),\xi(\eta^{\overline{A}}_{2i}(\overline{a_1},\dots, \overline{a_n})))\\
    &=\theta(\phi_i^A(\eta^{\overline{A}}_{2i-1}(\overline{a_1},\dots, \overline{a_n}),\eta^{\overline{A}}_{2i}(\overline{a_1},\dots, \overline{a_n}))).
\end{align*}
Therefore
\[
\phi^B(\xi^n(\overline{a_1},\dots,\overline{a_n}))
=\theta(\phi^A(\overline{a_1},\dots, \overline{a_n})).
\]
This means that the diagram \eqref{eq:isoclinismBis} is commutative.
\end{proof}

\begin{notation}
    For $X$ a skew brace and $x,y\in X$, let 
    \begin{align*}
        r(x,y)&= -y-x+x\circ y= [-y,-x+x\circ y]_++x*y,\\
        l(x,y)&=x\circ y-y-x= [x\circ y-y,-x]_++x*y.
    \end{align*}
\end{notation}

\begin{pro}
\label{pro: circlecomap}
For all skew braces $X$, the map 
\begin{equation*}
  \phi_{\circ}^X\colon \overline{X}^2\to X',\quad 
  (\overline{a},\overline{b})\mapsto [a,b]_{\circ},\\
\end{equation*}
is well-defined. In addition, if $A$ and $B$ are isoclinic skew braces
with isoclinism $(\xi,\theta)$, then 
    \begin{equation}
        \label{eq:isoclinismBis}
\begin{tikzcd}
	{\overline{A}^2} & {A'} \\
	{\overline{B}^2} & {B'}
	\arrow["{\phi_{\circ}^A}", from=1-1, to=1-2]
	\arrow["{\phi_{\circ}^B}"', from=2-1, to=2-2]
	\arrow["\xi\times\xi"', from=1-1, to=2-1]
	\arrow["\theta", from=1-2, to=2-2]
\end{tikzcd}
\end{equation}
commutes.
\end{pro}

\begin{proof}
    Let $X$ be a skew brace. 
    For all $a,b\in X$,
     \[
    [a,b]_{\circ}=[a,b]_+-l(b,a)-r(b\circ a,a'\circ b')+r(b,a'\circ b')+r(a,b\circ a'\circ b').
    \]
    Lemma \ref{lem: natrans} concludes the proof.
\end{proof}

\begin{pro}
    Let $A$ and $B$ be two skew braces. If $A$ and $B$ are isoclinic, 
    then $A_{+}$ is isoclinic to $B_{+}$ and $A_{\circ}$ is isoclinic to $B_{\circ}$.
\end{pro}

\begin{proof}
   Assume $A$ is isoclinic to $B$ with isoclinism $(\xi,\theta)$. Because of the commutativity of \eqref{eq:isoclinism}, the map $\theta$ restricts to an isomorphism $\theta_1: [A,A]_+\rightarrow [B,B]_+$ and $\xi(Z(A_+)/\Ann A)= Z(B_+)/\Ann B$. Therefore $\xi$ induces a group isomorphism $\xi_1: A_+/Z(A_+)\to B_+/Z(B_+)$ such that
    \[\begin{tikzcd}
    {(A_+/ Z(A_+))^2} & {A_+'} \\
	{(B_+/Z(B_+))^2} & {B_+'}
	\arrow[from=1-1, to=1-2]
	\arrow[from=2-1, to=2-2]
	\arrow["\xi\times\xi"', from=1-1, to=2-1]
	\arrow["\theta_1", from=1-2, to=2-2]
\end{tikzcd}
    \]
    commutes, where the horizontal maps are the classical commutator maps for groups. 
    Proposition \ref{pro: circlecomap} and
    a similar argument shows that $A_{\circ}$ is isoclinic to $B_{\circ}$.
\end{proof}

\begin{rem}
    If $A$ and $B$ are isoclinic skew braces, then 
    $A$ is of nilpotent type if and only if $B$ is of nilpotent type. 
\end{rem}

Recall that a skew brace $A$ is said to be \emph{two-sided}
if $(a+b)\circ c=a\circ c-c+b\circ c$ holds for all $a,b,c\in A$. 
In \cite{MR2278047}, Rump proved that radical rings are 
exactly 
two-sided skew braces with an abelian additive
group.

\begin{thm}
\label{thm:2sided}
    Let $A$ and $B$ be skew braces. If $A$ and $B$ are isoclinic, 
    then $A$ is two-sided if and only if $B$ is two-sided.
\end{thm}

\begin{proof}
    We claim that for all skew braces $X$
    \begin{equation*}
    \phi^X\colon (\overline{X})^3\to X',(\overline{a},\overline{b},\overline{c})\mapsto (a+b)\circ c-b\circ c+c -a\circ c,\\
    \end{equation*}
is well-defined. In addition, if $A$ and $B$ be isoclinic skew braces with isoclinism $(\xi,\theta)$. the following diagram
    \begin{equation}
    \label{eq:isoclinismTS}
    \begin{tikzcd}
	{\overline{A}^3} & {A'} \\
	{\overline{B}^3} & {B'}
	\arrow["{\phi^A}", from=1-1, to=1-2]
	\arrow["{\phi^B}"', from=2-1, to=2-2]
	\arrow["{\xi^3}"', from=1-1, to=2-1]
	\arrow["\theta", from=1-2, to=2-2]
\end{tikzcd}
\end{equation}
    commutes.
    This is a direct consequence of the Lemma \ref{lem: natrans} and the fact that for any elements $a,b,c$ of any brace, the following equation holds
    \[(a+b)\circ c-b\circ c+c -a\circ c= l(a+b,c)-l(b,c)-[a,b\circ c-c-b]_+-l(a,c)
    \]
    This concludes the proof.
\end{proof}
\begin{exa}
    Computer calculations using the database of 
    \cite{MR3647970} show that
    among the 
    47 skew braces of size eight, 
    there are 
    20 isoclinism classes. Moreover,
    there are eight isoclinism 
    classes of radical rings of size eight, 
    12 isoclinism classes of skew braces
    of abelian type of size eight and
    16 isoclinism classes of two-sided
    skew braces of size eight. 
    There are 101 skew braces of size 27 and  
    there are 
    38 isoclinism classes.
    See 
    Table \ref{tab:isoclinism} for other
    numbers. 
    \begin{table}[ht]
        \centering
        \begin{tabular}{c|cccc}
            size & radical rings & abelian type & two-sided & all \\
            \hline
            8 & 8 & 12 & 16 & 20\tabularnewline 
            27 & 10 & 13 & 25 & 38\tabularnewline
        \end{tabular}
        \caption{Number of isoclinism classes of skew braces.}
        \label{tab:isoclinism}
    \end{table}
    
\end{exa}

\begin{notation}
\label{not:conjugation}
Let $B$ be a skew brace. 
There is a canonical group homomorphism 
\[
B_+\rtimes_\rho B_{\circ}\to \Aut(B_+),
\quad
(a,b)\mapsto (c\mapsto a+\rho_b(c)-a),
\]
We write $\prescript{(a,b)}{}{}c=a+\rho_b(c)-a$.
\end{notation}

\begin{notation}
Let $A$ be a skew brace, denote respectively by $\overline{A_+}$ and $\overline{A_\circ}$ the groups $A_+/\Ann(A)$ and $A_\circ/\Ann(A)$.
\end{notation}

\begin{rem}
The group homomorphism $\rho\colon A_\circ\to\Aut(A_+)$ induces a
group homomorphism $\overline{\rho}\colon\overline{A_\circ} \to\Aut(\overline{A_+})$. One can check that the group $\overline{A_+}\rtimes_{\overline{\rho}}\overline{A_\circ}$ is isomorphic to $(A_+\rtimes_{\rho}A_\circ)/(\Ann(A)\times\Ann(A))$. Thus $\overline{A_+}\rtimes_{\overline{\rho}}\overline{A_\circ}$ acts canonically on $A_+$.
\end{rem}

\begin{rem}
If $A$ and $B$ are isoclinic skew braces, then $\overline{A_+}\rtimes_{\overline{\rho}}\overline{A_\circ}\simeq \overline{B_+}\rtimes_{\overline{\rho}}\overline{B_\circ}$.
\end{rem}

\begin{notation}
Let $A$ be a skew brace. Let $H$ be a subgroup
of $\overline{A_+}\rtimes_{\overline{\rho}} \overline{A_\circ}$. We call the orbit of an element $a\in A$ under the induced action of $H$ an $H$-orbit.
\end{notation}

\begin{thm}
\label{thm:Horbits}
Let $A$ and $B$ be isoclinic skew braces. Let $H$ 
be a subgroup of $\overline{A_+}\rtimes_{\overline{\rho}}\overline{A_\circ}$ 
and $K$ be the corresponding subgroup of  $\overline{B_+}\rtimes_{\overline{\rho}}\overline{B_\circ}$. For 
$c\in\Z_{\geq1}$, let 
$m_1$ (resp. $m_2$) be the number of $H$-orbits (resp. $K$-orbits) 
of size $c$. 
Then 
\[
m_1=m_2|A|/|B|.  
\]
\end{thm}

\begin{proof}
    Lemma \ref{lem: natrans} and the fact that 
    \[
    \prescript{(a,b)}{}{}c-c=[a,b\circ c-b]_++[b\circ c,-b]_++b*c
    \]imply that the map 
    \[
    \phi^X\colon \overline{X}^3\to X',\quad (\overline{a},\overline{b},\overline{c})\mapsto \prescript{(a,b)}{}{}c-c
    \]
    is well-defined for all skew brace $X$. 
    In addition, the diagram
    \begin{equation}
    \label{eq:isoclinismConj}
    \begin{tikzcd}
	{\overline{A}^3} & {A'} \\
	{\overline{B}^3} & {B'}
	\arrow["{\phi^A}", from=1-1, to=1-2]
	\arrow["{\phi^B}"', from=2-1, to=2-2]
	\arrow["{\xi^3}"', from=1-1, to=2-1]
	\arrow["\theta", from=1-2, to=2-2]
\end{tikzcd}
\end{equation}
commutes.

    An element $a\in A$ has an $H$-orbit of size $c$ if and only if the index of the subgroup $C(\overline{a})=\lbrace w\in H :\phi^A(w,\overline{a})=0\rbrace$ in $H$ is $c$. Let $S$ be 
    the subset of $\overline{A}$ that consists of elements $w$ such that $C(w)$ has index $c$ in $H$. If $\pi\colon A\to \overline{A}$
    is the canonical homomorphism, $\pi^{-1}(S)$ is the set of elements of $A$ that have an $H$-orbit of size $c$. Hence $m_1c=|\Ann A||S|$. Because of the commutativity of \eqref{eq:isoclinismConj}, one also has that $m_2c=|\Ann B||S|$. Hence the claim follows.
\end{proof}

\begin{rem}
    We use the notations of Theorem \ref{thm:Horbits}. 
    \begin{enumerate}
    \item Let 
    \begin{align*}
    H&=\{(-\overline{a},\overline{a}): a\in A\}
    \subseteq \overline{A_+}\rtimes_{\overline{\rho}}\overline{A_\circ},\\
    K&=\{(-\overline{b},\overline{b}): b\in B\}
    \subseteq \overline{B_+}\rtimes_{\overline{\rho}}\overline{B_\circ}.
    \end{align*}
    Note that $K$ is the subgroup of  $\overline{B_+}\rtimes_{\overline{\rho}}\overline{B_\circ}$
    corresponding to $H$ 
    by the isomorphism induced by isoclinism. 
    Then the $H$-orbits and the 
    $K$-orbits are, respectively, the orbits of the canonical actions $\lambda\colon A_\circ\to\Aut(A_+)$ and $\lambda\colon 
    B_\circ\to\Aut(B_+)$. 
    
    \item 
    Similarly, Theorem \ref{thm:Horbits} applies to the
    pair 
    $H=\{(\overline{0},\overline{a}): a\in A\}$ and 
    $K=\{(\overline{0},\overline{b}): b\in B\}$. In this case, 
    the $H$-orbits and the 
    $K$-orbits are, respectively, the orbits of the canonical actions $\rho\colon A_\circ\to\Aut(A_+)$ and $\rho\colon 
    B_\circ\to\Aut(B_+)$. 
    \end{enumerate}
\end{rem}

\begin{exa} 
Let $A$ be the skew brace $C_2\times C_4$ with multiplication given by
\[(x_1,y_1)\circ (x_2,y_2)=(x_1+x_2,y_1+y_2+2x_1y_2)
\]
and $B$ the skew brace $C_2\times C_4$ with multiplication given by
\[(x_1,y_1)\circ (x_2,y_2)=(x_1+x_2,y_1+y_2+2(x_1+y_2)x_2+2y_1y_2).
\]
In the skew brace $A$, 
\begin{gather*}
(x_1,y_1)* (x_2,y_2)=(0,2x_1y_2)   
\shortintertext{and in $B$}
(x_1,y_1)*(x_2,y_2)=(0,2(x_1+y_2)x_2+2y_1y_2).
\end{gather*}
Both $A$ and $B$ have commutator $C_2$ and annihilator quotient $C_2\times C_2$. In addition, $A_+ \simeq B_+$ and $A_\circ\simeq D_8\simeq B_\circ$ where $D_8$ is the dihedral group of order 8. However, $A$ and $B$ are not isoclinic as $A$ has four $\lambda$-orbits of size one and two of size two and $B$ has two $\lambda$-orbits of size one and three of size two. 
\end{exa}

In group theory, the notion of isoclinism is very convenient in the study of finite $p$-groups as these groups have non-trivial center and their commutator is a proper subgroup. This implies that isoclinism only depends on relations between groups of smaller order. However, there exist skew braces of prime-power size that have trivial annihilator.
This is not the case for two-sided skew braces. An interesting property of two-sided skew braces is that the multiplicative conjugation is an action by automorphism of the multiplicative group over the additive one. Using this,  one can extend the action defined earlier (see Notation \ref{not:conjugation}) for two-sided skew braces. Let $B$ be a two-sided skew brace. Since 
the underlying multiplicative group acts on itself by conjugation, 
we can consider the semidirect product $B_\circ \rtimes B_\circ$. 
The map 
\[
B_\circ\rtimes B_\circ \to \Aut(B_+),
\quad
(a,b)\mapsto (c\mapsto \rho_a(b\circ c\circ b')),
\]
defines an action by automorphisms of $B_\circ\rtimes B_\circ$ over $B_+$. The latter comes from the fact that $\rho_{b\circ a\circ b'}(b\circ c\circ b')=b\circ \rho_a(c)\circ b'$ for all $a,b,c\in B$. 
Thus 
we can consider the semi-direct product
$B_+\rtimes(B_\circ\rtimes B_\circ)$. Finally, straightforward computations show that the map
\[
B_+\rtimes (B_\circ\rtimes B_\circ) \to \Aut(B_+),
\quad
(a,b,c)\mapsto 
(d\mapsto a+\rho_b(c\circ d\circ c')-a),
\]
defines an action by automorphisms of $B_+\rtimes(B_\circ\rtimes B_\circ)$ over $B_+$. It is straightforward to see that the elements of $B$ whose orbits have size one are exactly the elements of the annihilator. Moreover, if $B$ is a two-sided skew brace of size $p^n$, the group $B_+\rtimes (B_\circ\rtimes B_\circ)$ has size $p^{3n}$. Thus the non-trivial orbits of the action have size a power of $p$. Let $n_1,\dots, n_m$ denote the sizes of the $m$ non-trivial orbits of $B$, then we have the following class equation
\[p^n= |\Ann(B)|+\sum_{i=1}^mn_i.
\]
Therefore $p$ divides $|\Ann(B)|$.
We have proved the following result:

\begin{pro}
    Let $p$ be a prime number and 
    $B$ be a two-sided skew brace of size $p^n$ 
    for some integer $n\geq1$. 
    Then $\Ann(B)$ is non-trivial. 
\end{pro}

\begin{pro}
    Let $p$ be a prime number and 
    $B$ be a two-sided skew brace of size $p^n$ 
    for some integer $n\geq1$. 
    Then $B'$ is a proper ideal of $B$.
\end{pro}

\begin{proof}
    We proceed by induction on $n$. If $n=1$, then $B$ is the trivial skew brace $C_p$. Assume now that $n\geq 2$. Since $|B/\Ann(B)|<p^n$, it follows by the induction hypothesis that $\overline{B}'=(B'+\Ann(B))/\Ann(B)$ is a proper sub skew brace of $\overline{B}$. Thus $B'+\Ann(B)\subsetneq B$.
\end{proof}

\section{An application to the Yang--Baxter equation}
\label{section:YB}

A \emph{set-theoretic solution} to the 
Yang--Baxter equation (YBE) 
is a pair $(X,r)$, where $X$ is a set and 
$r\colon
X\times X\to X\times X$ is a bijective map 
such that
\[
    (r\times\id)(\id\times r)(r\times\id)=(\id\times r)(r\times\id)(\id\times r).
\]

By convention, we will consider 
finite \emph{non-degenerate} solutions, 
that is solutions $(X,r)$,
where $X$ is a finite set and 
\[
r(x,y)=(\sigma_x(y),\tau_y(x))
\]
where the maps $\sigma_x\colon X\to X$ and $\tau_x\colon X\to X$
are bijective for every $x\in X$. 

If $(X,r)$ is a solution, 
there is an equivalence relation
on $X$ given~by
\[
x\sim y\Longleftrightarrow \sigma_x=\sigma_y\text{ and }\tau_x=\tau_y.
\]
This equivalence relation induces a solution 
$\Ret(X,r)$ on the set of equivalence classes. 
The solution $\Ret(X,r)$ 
is called the \emph{retraction}
of $(X,r)$. 

A solution $(X,r)$ is said to be 
\emph{multipermutation} if
there exists an integer $m\geq1$ 
such that $|\Ret^m(X,r)|=1$, 
where $\Ret^{1}(X,r)=\Ret(X,r)$
and 
\[
\Ret^{k}(X,r)=\Ret(\Ret^{k-1}(X,r))
\]
for $k\geq2$.

The \emph{permutation group} of $(X,r)$ if the group
$\mathcal{G}(X,r)=\langle \sigma_x,\tau_x:x\in X\rangle$. 
The permutation group of $(X,r)$ is a skew brace \cite{MR3835326}. 

\begin{defn}
    Let $(X,r)$ and $(Y,s)$ be solutions to the YBE. We say that 
    $(X,r)$ and $(Y,s)$ are \emph{permutation isoclinic} if 
    the skew braces $\mathcal{G}(X,r)$ and 
    $\mathcal{G}(Y,s)$ are isoclinic. 
\end{defn}

The following result follows from 
\cite[Proposition 2.17]{MR3957824}, 
\cite[Theorem 2.20]{MR3957824} and
\cite[Theorem 4.13]{MR4457900}. 

\begin{lem}
    Let $(X,r)$ be a finite non-degenerate solution to the YBE. 
    The following statements are equivalent:
    \begin{enumerate}
        \item $(X,r)$ is multipermutation.
        \item $G(X,r)$ is right nilpotent of nilpotent type.
        \item $\mathcal{G}(X,r)$ is right nilpotent of nilpotent type.
    \end{enumerate}
\end{lem}

\begin{thm}
\label{thm:YB}
    Let $(X,r)$ and $(Y,s)$ be permutation isoclinic solutions to the YBE.  
    Then $(X,r)$ is multipermutation if and only if $(Y,s)$ is multipermutation. 
\end{thm}

\begin{proof}
    If $(X,r)$ is multipermutation, then 
    $\mathcal{G}(X,r)$ is right
    nilpotent. Thus $\mathcal{G}(Y,s)$ is right nilpotent
    and hence $(Y,s)$ is multipermutation. 
\end{proof}

We conclude the paper with concrete examples of
involutive solutions up to permutation isoclinism. Recall that solution $(X,r)$ is
said to be 
\emph{involutive} if $r^2=\id$. If $(X,r)$ is involutive, 
then 
\[
\tau_y(x)=\sigma^{-1}_{\sigma_x(y)}(x)
\]
for all $x,y\in X$. 

\begin{exa}
\label{exa:size4}
    There are four permutation isoclinism classes
    of involutive 
    solutions of size four. Let $X=\{1,2,3,4\}$. The following 
    list provides a complete set of representatives over the set $X$:
    \begin{enumerate}
        \item The flip $(x,y)\mapsto (y,x)$. 
        \item $\sigma_1=\sigma_2=\id$, $\sigma_3=(34)$ and $\sigma_4=(12)(34)$.
        \item $\sigma_1=(34)$, $\sigma_2=(1324)$, $\sigma_3=(1423)$ and $\sigma_4=(1,2)$. 
        \item $\sigma_1=(12)$, $\sigma_2=(1324)$, $\sigma_3=(34)$ and $\sigma_4=(1423)$.
    \end{enumerate}
\end{exa}

\begin{rem}
    Permutation isoclinism of solutions does not preserve 
    indecomposability. For example, 
    let $X=\{1,2,3,4\}$ and $\sigma=(1234)$. Then $(X,r)$,
    where 
    \[
    r(x,y)=(\sigma(y),\sigma^{-1}(x)),
    \]
    is indecomposable
    and 
    $\mathcal{G}(X,r)$ is the trivial skew brace 
    over the cyclic group $C_4$. It follows that  
    $(X,r)$ is isoclinic to the flip over $X$, as both
    solutions have isoclinic permutation braces (note that
    the permutation group of the flip is the trivial group). 
\end{rem}

\begin{rem}
    Permutation isoclinism of solutions does not preserve 
    the multipermutation level. For example, 
    let $X=\{1,2,3,4\}$ and $\sigma_1=\sigma_2=\sigma_3=\id$, 
    $\sigma_4=(23)$. Then $(X,r)$ 
    has multipermutation level two. Moreover, 
    $\mathcal{G}(X,r)$ is the trivial 
    skew brace over the cyclic group $C_2$. Hence the 
    solution $(X,r)$ is permutation isoclinic to the flip over $X$. 
\end{rem}

\begin{exa}
    There are six permutation isoclinism classes
    of involutive 
    solutions of size five. Let $X=\{1,2,3,4,5\}$. The following 
    list provides a complete set of representatives over $X$:
    \begin{enumerate}
        \item The flip $(x,y)\mapsto (y,x)$. 
        \item $\sigma_1=\sigma_2=\sigma_3=\id$, $\sigma_4=(45)$, $\sigma_5=(23)(45)$.
        \item $\sigma_1=\sigma_2=\sigma_3=\id$, $\sigma_4=(23)(45)$, $\sigma_5=(12)(45)$.
        \item $\sigma_1=\id$, $\sigma_2=(45)$, $\sigma_3=(2435)$, $\sigma_4=(2534)$
        and $\sigma_5=(23)$. 
        \item $\sigma_1=\id$, $\sigma_2=(23)$, $\sigma_3=(2435)$, $\sigma_4=(45)$ 
        and $\sigma_5=(2534)$. 
        \item $\sigma_1=\sigma_2=(45)$, $\sigma_3=(14)(25)$ and $\sigma_4=\sigma_5=(12)$. 
    \end{enumerate}
    Note that flips of size four and five
    are permutation isoclinic. 
    The second solution is permutation
    isoclinic to the second solution of 
    Example \ref{exa:size4}. 
    The fourth solution is permutation isoclinic to 
    the third solution 
    of Example \ref{exa:size4}. The fifth solution
    is permutation isoclinic to the fourth solution
    of Example~\ref{exa:size4}. This is a complete 
    set of permutation isoclinisms between solutions of sizes 
    four and five. 
\end{exa}

\section*{Acknowledgements}

Vendramin is supported in part by OZR3762 of Vrije Universiteit Brussel. The authors thank the referee for useful comments
and suggestions. 

\bibliographystyle{abbrv}
\bibliography{refs}

\end{document}